\documentclass[12pt,twoside]{article}
\usepackage{amsmath,amssymb}
\usepackage{graphicx}

\allowdisplaybreaks[4]

\newtheorem{lm}{Lemma}[section]
\newtheorem{thm}{Theorem}[section]
\newtheorem{rmk}{Remark}[section]

\newcounter{saveeqn}%

\makeatletter
\evensidemargin
\oddsidemargin
\makeatother

\title{\Large\bf
Asymptotic expansions of characteristic orbits of planar real analytic vector fields
\thanks{Supported by NSFC \#12101087 and NSFSPC \#2024NSFSC1400.}
}

\author{
{\sc Jun Zhang}
\footnote{Corresponding Author. E-mail address: mathzhangjun@163.com}
\\
$^a${\small School of Mathematical Sciences \& Sichuan Geomath Key Lab}\\
{\small Chengdu University of Technology, Sichuan 610059, P. R. China}
}
\date{}

\begin{document}
\maketitle

\begin{abstract}
The well-known Newton-Puiseux Theorem states that
each real branch of a planar real analytic curve admits a Puiseux expansion.
We generalize this result to characteristic orbit of an isolated singularity of
a planar real analytic vector field and prove that
each characteristic orbit has a `power-log' expansion.

\vskip 0.2cm



\end{abstract}

\baselineskip 15pt   
\parskip 10pt         

\thispagestyle{empty}
\setcounter{page}{1}

\section{Introduction and main result}
\setcounter{equation}{0}
\setcounter{lm}{0}
\setcounter{thm}{0}
\setcounter{rmk}{0}
\setcounter{df}{0}
\setcounter{cor}{0}
\setcounter{pro}{0}

The well-known Newton-Puiseux Theorem (\cite[Theorem 4.2.7]{KP02}) states that
each real branch of a planar real analytic curve locally has a Puiseux expansion
(given by a convergent fractional power series) no matter what analytic coordinates are chosen.

In this paper,
we consider a planar real analytic vector field ${\cal V}$ near a point $P$.
When $P$ is a singular point, we always assume that it is isolated.
As indicated in \cite[p.79, Theorem~3.10]{ZZF} or \cite[p.73, Lemma~1]{SC},
if there is an orbit connecting with $P$,
then it connects with $P$ either spirally or in a fixed direction.
The orbit in the latter case is referred to as a \textit{characteristic orbit}.
Let $(x,y)$ be an analytic coordinates near $P$ and
assume that under which
\begin{align}
{\cal V}=X(x,y)\partial_x+Y(x,y)\partial_y,
\label{equ:XY}
\end{align}
where $X$ and $Y$ are analytic functions at the origin $O:(0,0)$.
Suppose that the vector field ${\cal V}$ has a characteristic orbit $\Gamma$.
Then,
the orbit $\Gamma$ can be viewed as the image of a function $y=y(x)$
for small $x\ge 0$ or $x\le 0$
unless $\Gamma$ lies on the $y$-axis.
Actually,
the Newton-Puiseux Theorem (\cite[Theorem 4.2.7]{KP02}) implies that there are finitely many
vertical isoclines and horizontal ones,
determined by equations $X(x,y)=0$ and $Y(x,y)=0$, respectively.
Since a non-zero function given by a Puiseux series is locally monotonous,
each isocline is either an orbit lying on the axes or without contact.
Moreover,
those isoclines divide a small neighborhood of the origin $O$
into finitely many disjoint subregions,
in each of which both $X(x,y)$ and $Y(x,y)$ have definite signs.
It follows that
the characteristic orbit $\Gamma$ either lies on the $y$-axis or
is the image of a function $y=y(x)$ for small $x\ge 0$ or $x\le 0$.
So similar to the Puiseux expansion of real branches of a planar analytic curve,
we are interested in the asymptotic expansions
(independent of the choice of analytic coordinates)
of characteristic orbits.
Since each characteristic orbit is an invariant manifold,
this problem is naturally related to the smoothness of invariant manifold.
It is also helpful to construct a generalized normal sector \cite{T-Z}
when investigating orbit near an exceptional direction of a singularity.
Moreover,
asymptotic expansions of characteristic orbits can be employed
to compute the box dimension of orbits on separatrices generated by the unit time map,
which leads to a fractal classification of singularities, see \cite{DHVZ}.

When $O$ is non-singular,
the Cauchy's Theorem (\cite[p.34, Theorem~8.1]{CL87}) ensures that
there is a unique characteristic orbit connecting with $O$ and
it has an analytic parametrization $y=y(x)$ under some appropriate analytic coordinates.
Then
Puiseux expansion is the analytic coordinates free
asymptotic expansion of the characteristic orbit
in this non-singular case.
However,
things get complicated when $O$ is a singular point.
We first consider the hyperbolic cases:
\begin{description}
  \item[(H1)]
  When $O$ is a hyperbolic focus, there are no characteristic orbits.

  \item[(H2)]
  When $O$ is a hyperbolic saddle, then characteristic orbits all lie
  on the analytic stable manifold and unstable manifold.
  So the analytic coordinates free asymptotic expansion is a Puiseux series.

  \item[(H3)]
  When $O$ is a hyperbolic node,
  consider the analytic vector field
  \begin{align*}
  x \partial_x+(\lambda y+\cdots) \partial_y,
  \end{align*}
  where $\lambda>0$ and $\cdots$ represents higher order terms.
  Picard (\cite{P78,P84}) and Poincar\'e (\cite{Poincare}) proved that
  each characteristic orbit (not lying on the $y$-aixs) of the above vector field
  is a convergent series in $x$ and $x^\lambda$ for non-integer $\lambda$ but
  a convergent series in $x$ and $x\ln x$ for integer $\lambda$.
  However, this expansion is not an analytic coordinates free expansion.
\end{description}
Next,
we consider those non-hyperbolic cases.
Clearly,
there is no need to consider center-type singularities
since no characteristic orbit exists.
For a semi-hyperbolic singularity (exact one zero eigenvalue),
Theorem~2.19 given in \cite{DLA} ensures that
each characteristic orbit either lies on the unique strong (un)-stable manifold,
which is $C^\omega$,
or a center manifold, which is $C^\infty$.
Then the analytic coordinates free expansion is either a convergent or formal Puiseux series.
For nilpotent singularities and full-null ones,
i.e., the Jacobian matrices of which are nilpotent and zero matrices respectively,
Brezovskaya \cite{Berezov} used Newton polygon to determine
the principle part of the asymptotic expansion of a characteristic orbit.
However, general expansions remain to be answered.
Consequently,
we need to determine analytic coordinates free asymptotic expansions for
hyperbolic nodes, nilpotent singularities and full-null singularities.
Such expansions beyond Puiseux series but transseries will be involved
as shown in \cite{P78,P84,Poincare}.

Transseries plays an important role in the study of dynamical systems.
Let ${\cal P}$ be a hyperbolic monodromic polycycle of a planar analytic vector field,
and consider a transversal section $\Sigma$ for the polycycle ${\cal P}$
with a local coordinate $x$ along the section
whose origin is exact the intersection of $\Sigma$ and ${\cal P}$.
Then the Poincar\'e return map $D$ and
the return time function $T$ can be defined on $\Sigma$.
Dulac \cite{Dul}, \'{E}calle \cite{Ecalle} and Il'yashenko \cite{Il90,Il91} proved that
the Poincar\'e return map $D(x)$ has at the origin the following asymptotic expansion
\begin{align*}
\widehat{D}(x)=c_0x^{\lambda_0}+\sum_{i\ge 1} x^{\lambda_i}P_i(\ln x),~~~x>0,
\end{align*}
where $c_0>0$,
$(\lambda_0,\lambda_1,\cdots)$ is a strictly increasing sequence of positive
real numbers tending to infinity and each $P_i$ is a polynomial with real coefficients.
On the other hand,
Saavedra \cite{MS06} proved that
return time function $T(x)$ has the following asymptotic expansion
\begin{align*}
\widehat{T}(x)=\sum_{i\ge 0} x^{\lambda_i}P_i(\ln x),~~~x>0,
\end{align*}
which is similar to the above expansion of $D(x)$
but allows negative powers and non-constant leading coefficient $P_0$.
However,
iterated logrithms and exponentials will be involved for non-hyperbolic polycycles.
Recently,
normal form and classification of transseries have been considered, see \cite{MR21, MRRZ, PRRS}.

Let $\mathbb{R}[\![x]\!]$ (and $\mathbb{R}\{x\}$) be 
the ring of formal (convergent) power series in $x$ with coefficients in $\mathbb{R}$. 
For asymptotic expansions of characteristic orbits of a planar analytic vector field,
we obtain the following result.

\begin{thm}
Suppose that $y=y(x)$ for small $x>0$ is a characteristic orbit of vector field \eqref{equ:XY}.
Then either
\begin{description}

\item[(i)]
$y(x)\in\mathbb{R}[\![x^{\frac{1}{n}}]\!]$ for some $n\in\mathbb{N}_+$, or

\item[(ii)]
$y(x)=\sum_{i\ge 1} c_i x^{\lambda_i}$,
where $(\lambda_1,\lambda_2,\cdots)$ is a strictly increasing sequence of positive
real numbers tending to infinity
, or

\item[(iii)]
$y(x)=\sum_{i+j\ge 1} x^{\frac{i}{n}}{\boldsymbol{\ell}}^{\frac{j}{n}}
P_{ij}(\ln x,\ln \boldsymbol{\ell})$,
where $\boldsymbol{\ell}:=\frac{-1}{\ln x}$ and
$P_{ij}$\,s are polynomials.
\end{description}
\label{thm:O}
\end{thm}

\section{Proof of the main result}
\setcounter{equation}{0}
\setcounter{lm}{0}
\setcounter{thm}{0}
\setcounter{rmk}{0}
\setcounter{df}{0}
\setcounter{cor}{0}
\setcounter{pro}{0}

By the theorem of desingularization (\cite[Theorem~3.3]{DLA}),
an isolated singularity of the analytic vector field ${\cal V}$ can be decomposed into
a finite number of elementary singularities
(having at least one non-zero eigenvalue).
As indicated in the introduction,
for elementary singularities,
the only unknown case for asymptotic expansions is the hyperbolic node case,
which involves resonance.
According to normal form theory (\cite[Chapter~2]{DLA}),
under an appropriate analytic coordinates and a non-zero constant being eliminated,
a vector field with a hyperbolic node at the origin is one of the following forms
\begin{equation*}
x \partial_x+y \partial_y,~~~
x \partial_x+(x+y) \partial_y,~~~
x \partial_x+\lambda y \partial_y,~~~
x \partial_x+(py+b x^p) \partial_y,
\end{equation*}
where $\lambda\in\mathbb{R}_+\!\setminus\!\mathbb{N}$,
$p\in\mathbb{N}_+\backslash\{1\}$ and $b\in\mathbb{R}$.
Clearly,
those characteristic orbits not lying on the $y$-axis of the above vector fields
are given by
\begin{equation}
y=cx,~~~
y=x(c+\ln|x|),~~~
y=c|x|^\lambda,~~~
y=x^p(c+b\ln|x|),
\label{node-phi}
\end{equation}
respectively, where $c\in\mathbb{R}$.

\subsection{Desingularization}

According to the desingularization process (\cite[Chapter~3]{DLA}),
there exist local coordinates $(x_k,y_k)$ for $k=0, 1, ..., m-1$ satisfying that
\begin{align}
(x_{m-1},y_{m-1})\stackrel{\pi_{m-1}}{\longrightarrow}(x_{m-2},y_{m-2})
\stackrel{\pi_{m-2}}{\longrightarrow}
\cdots
\stackrel{\pi_2}{\longrightarrow}
(x_1,y_1)\stackrel{\pi_1}{\longrightarrow}(x_0,y_0)=(x,y),
\end{align}
where $\pi_k\in\{{\cal B}_v,{\cal B}_h,{\cal T}_v,{\cal T}_h\}$ with
\begin{equation*}
\begin{aligned}
&
{\cal B}_v:
\left\{
\begin{array}{llll}
x_{k-1}=x_k,
\\
y_{k-1}=x_ky_k,
\end{array}
\right.
&&
{\cal B}_h:
\left\{
\begin{array}{llll}
x_{k-1}=y_k x_k,
\\
y_{k-1}=y_k,
\end{array}
\right.
\\
&
{\cal T}_v:
\left\{
\begin{array}{llll}
x_{k-1}=x_k,
\\
y_{k-1}=y_k+y_k^*,
\end{array}
\right.
&&
{\cal T}_h:
\left\{
\begin{array}{llll}
x_{k-1}=x_k+x_k^*,
\\
y_{k-1}=y_k
\end{array}
\right.
\end{aligned}
\end{equation*}
and neither $x_k^*$ nor $y_k^*$ is zero;
and that for $k=0,1,...,m-2$
the map $\pi_{k+1}^{-1}$ brings
the characteristic orbit $\Gamma_k$ of the vector field ${\cal V}_k$
connecting with the point $P_k$ into
the characteristic orbit $\Gamma_{k+1}$ of the vector field ${\cal V}_{k+1}$
connecting with the point $P_{k+1}$,
as depicted by Fig.~\ref{DeGamma},
where
$P_k$ is either the point $(0,0)$, $(0,y^*_k)$, or $(x^*_k,0)$,
$P_{k+1}$ is either the point $(0,0)$, $(0,y^*_{k+1})$, or $(x^*_{k+1},0)$,
and ${\cal V}_{k+1}$ is the pull back of ${\cal V}_k$
(some common factor could be eliminated).
Note that if $\pi_k={\cal T}_v$ (resp. ${\cal T}_h$)
then $\pi_{k-1}={\cal B}_v$ (resp. ${\cal B}_h$).
We can assume without loss of generality that
the characteristic orbit $\Gamma_{m-1}$ connects with
the origin $O_{m-1}:(0,0)$ in the $(x_{m-1},y_{m-1})$-plane;
otherwise, a more translation ${\cal T}_v$ or ${\cal T}_h$ can be performed.
Then the point $O_{m-1}$ is either nonsingular or
an elementary singularity of the vector field ${\cal V}_{m-1}$.

\begin{figure}[h]
  \centering
  \includegraphics[height=2.2in]{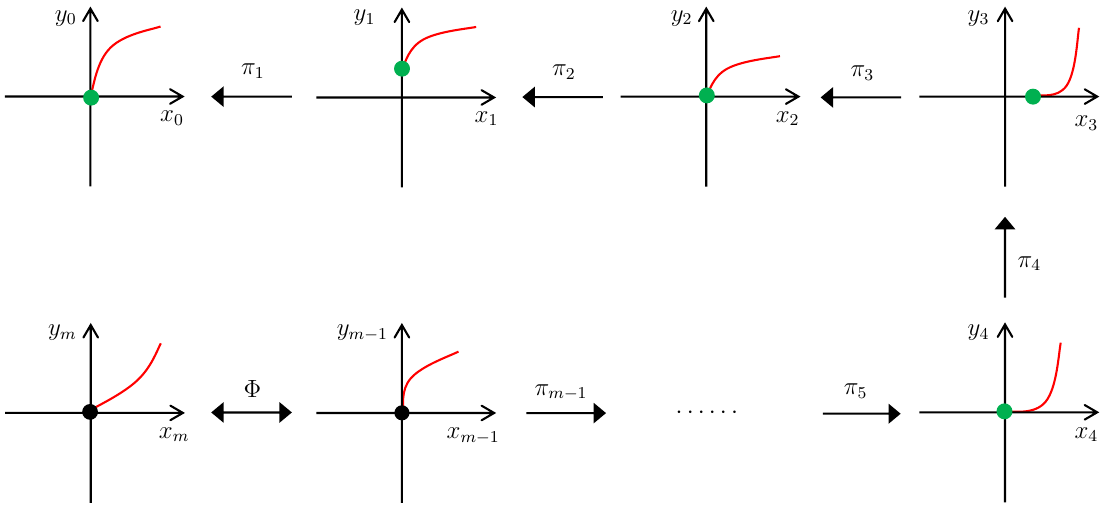}\\
  \caption{Desingularization process along characteristic orbit $\Gamma$.}
  \label{DeGamma}
\end{figure}

When $O_{m-1}$ is either nonsingular or a hyperbolic saddle,
there is an analytic coordinates change
\begin{align}
\Phi:
\left(
\begin{array}{cc}
x_m
\\
y_m
\end{array}
\right)
\to
\left(
\begin{array}{cc}
x_{m-1}
\\
y_{m-1}
\end{array}
\right)
=
\left(
\begin{array}{cc}
\phi_1(x_m,y_m)
\\
\phi_2(x_m,y_m)
\end{array}
\right)
\label{phi-n}
\end{align}
preserving the origin and straightening the characteristic orbit $\Gamma_{m-1}$
into the positive $x_m$-axis.

When $O_{m-1}$ is a hyperbolic node,
we see from \eqref{node-phi} that
under an analytic coordinates change \eqref{phi-n},
the orbit $\Gamma_{m-1}$ becomes one of the following three curves
$$
y_m=0,~~~y_m=cx_m^\lambda~~~\mbox{and}~~~y_m=x_m^p(c+b\ln x_m)
$$
for $x_m\ge 0$,
where $b, c\in\mathbb{R}$, $\lambda\in\mathbb{R}_+\!\setminus\!\mathbb{N}$
and $p\in\mathbb{N}_+$.

When $O_{m-1}$ is semi-hyperbolic,
if $\Gamma_{m-1}$ lies on the strong (un)-stable manifold,
then the orbit $\Gamma_{m-1}$ can be straightened into the positive $x_m$-axis
by an analytic coordinates change \eqref{phi-n} as above.
On the contrary,
if $\Gamma_{m-1}$ lies on a center manifold, which is $C^\infty$,
then there is a formal coordinates change \eqref{phi-n}
straightening the orbit $\Gamma_{m-1}$ into the positive $x_m$-axis.

As a consequence of the above three cases,
an analytic or formal coordinates change \eqref{phi-n} brings
the orbit $\Gamma_{m-1}$ into the curve
\begin{align}
y_m=\varphi(x_m)=0
\label{ynxn1}
\end{align}
for $x_m\ge 0$,
or an analytic coordinates change \eqref{phi-n} brings
the orbit $\Gamma_{m-1}$ into the curve
\begin{align}
&y_m=\varphi(x_m)=cx_m^\lambda,~~~\mbox{or}
\label{ynxn2}
\\
&y_m=\varphi(x_m)=x_m^p(c+b\ln x_m)
\label{ynxn3}
\end{align}
for $x_m\ge 0$,
where $b, c\in\mathbb{R}$, $\lambda\in\mathbb{R}_+\!\setminus\!\mathbb{N}$ and $p\in\mathbb{N}_+$.
After blowing down,
we obtain in the following lemma
the parametric representation of the characteristic orbit $\Gamma_0=\Gamma$.

\begin{lm}
The characteristic orbit $\Gamma$ of vector field \eqref{equ:XY} admits
the following parametric representation
\begin{equation}
\left\{
\begin{array}{llll}
x={\cal X}(x_m)
:=\alpha_{r_1,s_1} T_1^{r_1} T_2^{s_1}
+\sum_{i> r_1,j> s_1} \alpha_{i,j} T_1^i T_2^j,
\\
y={\cal Y}(x_m)
:=\beta_{r_2,s_2} T_1^{r_2} T_2^{s_2}+\sum_{i> r_2,j> s_2} \beta_{i,j} T_1^i T_2^j,
\end{array}
\right.
\label{xxnyyn}
\end{equation}
polynomials in $T_1$ and $T_2$,
where
$r_i$ and $s_i$ are non-negative integers with $r_i+s_i\ge 1$,
$T_i:=\phi_i(x_m,\varphi(x_m))$ with $\phi_i$ given in \eqref{phi-n} and
$\varphi$ is one of the forms \eqref{ynxn1}-\eqref{ynxn3} for $i=1,2$,
and $\alpha_{r_1,s_1}\beta_{r_2,s_2}\ne 0$.
\label{lm:par}
\end{lm}

{\bf Proof.}
Note that $\Gamma_m$ has the representation $y_m=\varphi(x_m)$,
where $\varphi$ is one of the forms \eqref{ynxn1}-\eqref{ynxn3}.
Then $\Gamma_{m-1}$ admits the parametric representation
\begin{equation*}
\left\{
\begin{array}{llll}
x_{m-1}=\phi_1(x_m,\varphi(x_m))=T_1,
\\
y_{m-1}=\phi_2(x_m,\varphi(x_m))=T_2,
\end{array}
\right.
\end{equation*}
the same form as \eqref{xxnyyn}.
Suppose that $\Gamma_k$, $k\in\{2,...m-1\}$, can be expressed as
\begin{equation*}
\left\{
\begin{array}{llll}
x_k=a_{i_1,j_1} T_1^{i_1} T_2^{j_1}+o(T_1^{i_1},T_2^{j_1}),
\\
y_k=b_{i_2,j_2} T_1^{i_2} T_2^{j_2}+o(T_1^{i_2},T_2^{j_2}),
\end{array}
\right.
\end{equation*}
the same form as \eqref{xxnyyn},
where $o(T_1^{i_1},T_2^{j_1})$ represents those terms
$\sum_{i> i_1,j> j_1} a_{i,j} T_1^i T_2^j$
for simplicity.
There are four cases:
{\bf(i)} $\pi_k={\cal B}_v$,
{\bf(ii)} $\pi_k={\cal B}_h$,
{\bf(iii)} $\pi_k={\cal T}_v~\mbox{and}~\pi_{k-1}={\cal B}_v$, and
{\bf(iv)} $\pi_k={\cal T}_h~\mbox{and}~\pi_{k-1}={\cal B}_h$.
In case {\bf(i)},
the orbit $\Gamma_{k-1}$ can be expressed as
\begin{equation*}
\left\{
\begin{array}{lllll}
x_{k-1}=x_k=a_{i_1,j_1}T_1^{i_1}T_2^{j_1}+o(T_1^{i_1},T_2^{j_1}),
\\
y_{k-1}=x_ky_k=a_{i_1,j_1}b_{i_2,j_2}T_1^{i_1+i_2}T_2^{j_1+j_2}
+o(T_1^{i_1+i_2},T_2^{j_1+j_2}),
\end{array}
\right.
\end{equation*}
the same form as \eqref{xxnyyn}.
Similarly,
in case {\bf(ii)},
the representation of the orbit $\Gamma_{k-1}$ is also of the same form as \eqref{xxnyyn}.
In case {\bf(iii)},
the orbit $\Gamma_{k-2}$ can be expressed as
\begin{equation*}
\left\{
\begin{array}{llll}
x_{k-2}=x_{k-1}=x_k=a_{i_1,j_1}T_1^{i_1}T_2^{j_1}+o(T_1^{i_1},T_2^{j_1}),
\\
y_{k-2}=x_{k-1}y_{k-1}=x_k(y_k+y_k^*)
=a_{i_1,j_1}y_k^*T_1^{i_1}T_2^{j_1}+o(T_1^{i_1},T_2^{j_1}),
\end{array}
\right.
\end{equation*}
the same form as \eqref{xxnyyn}.
Similarly,
in case {\bf(iv)},
the representation of the orbit $\Gamma_{k-2}$ is also of the same form as \eqref{xxnyyn}.
Since $\pi_1\in\{{\cal B}_v,{\cal B}_h\}$,
similar to the above,
the orbit $\Gamma=\Gamma_0$ has the parametric representation \eqref{xxnyyn}
and the proof is completed.
\qquad$\Box$

\begin{lm}
The function ${\cal X}$ given in Lemma~\ref{lm:par} satisfies one of the following cases:
\begin{description}

\item[(C1)]
${\cal X}(u)\in\mathbb{R}[\![u]\!]~\mbox{or}~\mathbb{R}\{u\}$,

\item[(C2)]
${\cal X}(u)\in\mathbb{R}\{u,u^\lambda\}$,

\item[(C3)]
${\cal X}(u)=c_k u^k+\sum_{i\ge k+1} u^i P_i(\ln u)$,

\item[(C4)]
${\cal X}(u)=u^k P_k(\ln u)+\sum_{i\ge k+1} u^i P_i(\ln u)$,
\end{description}
where $\lambda\in\mathbb{R}_+\setminus\mathbb{N}$,
$k$ is a positive integer, $c_k\ne0$, $P_i$\,s are polynomials and $\deg P_k\ge 1$.
\label{lm:XC1-4}
\end{lm}

\begin{rmk}
{\rm
The function ${\cal Y}$ given in Lemma~\ref{lm:par} also satisfies
one of the above 4 cases as ${\cal X}$.
Moreover,
if ${\cal X}$ satisfies {\bf(C1)} (or {\bf(C2)}),
then ${\cal Y}$ satisfies {\bf(C1)} (or {\bf(C2)}).
If ${\cal X}$ satisfies {\bf(C3)} (or {\bf(C4)}),
then ${\cal Y}$ satisfies either {\bf(C3)} or {\bf(C4)}.
}
\label{rmk:Y}
\end{rmk}

{\bf Proof.}
Note that $T_r(u)=\phi_r(u,\varphi(u))$ for $r=1,2$ and
$\varphi(u)$ takes one of the forms \eqref{ynxn1}-\eqref{ynxn3}.
When $\varphi(u)\equiv 0$,
we have $T_r(u)=\phi_r(u,0)\in\mathbb{R}[\![u]\!]$ or $\mathbb{R}\{u\}$
since $\phi_r(x,y)\in \mathbb{R}[\![x,y]\!]$ or $\mathbb{R}\{x,y\}$.
The fact that ${\cal X}$ is a polynomial in $T_1$ and $T_2$ yields that
${\cal X}(u)\in\mathbb{R}[\![u]\!]$ or $\mathbb{R}\{u\}$.
When $\varphi(u)=cu^\lambda$,
we have $\phi_r(x,y)\in \mathbb{R}\{x,y\}$ and therefore
$T_r(u)=\phi_r(u,cu^\lambda)\in \mathbb{R}\{u,u^\lambda\}$,
which implies that ${\cal X}(u)\in \mathbb{R}\{u,u^\lambda\}$ for the same reason as above.
So we obtain the first two possibilities {\bf(C1)} and {\bf(C2)}.

When $\varphi(u)=u^p(c+b\ln u)$,
we have $\phi_r(x,y)\in \mathbb{R}\{x,y\}$.
Assume that
$$
\phi_r(x,y)=\sum_{i+j\ge 1} \mu^{(r)}_{i,j} x^i y^j
$$
and let $\ell_r:=\min\{i:\mu^{(r)}_{i,0}\ne 0\}$ for $r=1,2$.
Since $\Phi=(\phi_1,\phi_2)$ is a coordinates change,
we have $\mu^{(1)}_{1,0}\mu^{(2)}_{0,1}-\mu^{(2)}_{1,0}\mu^{(1)}_{0,1}\ne 0$,
which implies that $\min\{\ell_1,\ell_2\}=1$.
One can compute that for $r=1,2$,
\begin{align*}
T_r(u)=
\left\{
\begin{array}{llll}
\mu^{(r)}_{\ell_r,0} u^{\ell_r}+\sum_{i>\ell_r} u^i Q_i^{(r)}(\ln u)
&\mbox{if}~1\le \ell_r\le p-1,
\\
u^pQ^{(r)}_p(\ln u)+\sum_{i>p} u^i Q_i^{(r)}(\ln u)
&\mbox{if}~\ell_r\ge p,
\end{array}
\right.
\end{align*}
where $Q^{(r)}_i$\,s are polynomials and $\deg Q^{(r)}_p\ge 1$.
Substituting the above expressions of $T_1(u)$ and $T_2(u)$ into
the expression \eqref{xxnyyn} of ${\cal X}$,
we obtain that the expansion of ${\cal X}$ satisfies {\bf(C3)} or {\bf(C4)}.
Thus the proof of Lemma~\ref{lm:XC1-4} is completed.
\qquad$\Box$

\subsection{Asymptotic expansion of the inverse function}

\begin{lm}
The asymptotic expansion of the inverse function ${\cal X}^{-1}(x)$
satisfies one of the three cases given in Theorem~\ref{thm:O}.
\label{lm:inverse}
\end{lm}

{\bf Proof.}
In case \textbf{(C1)} given in Lemma~\ref{lm:XC1-4},
the equation
\begin{align}
x-{\cal X}(u)=0
\label{equ:xX}
\end{align}
is a formal or analytic equation.
By the Newton-Puiseux Theorem (\cite[Theorem 4.2.7]{KP02}),
the solution $u={\cal X}^{-1}(x)$ is a formal or convergent Puiseux series,
i.e., ${\cal X}^{-1}(x)\in\mathbb{R}[\![x^{\frac{1}{n}}]\!]$ or
$\mathbb{R}\{x^{\frac{1}{n}}\}$
for some positive integer $n$,
which satisfies case {\bf(i)} or {\bf(ii)} of Theorem~\ref{thm:O}.

In case \textbf{(C2)},
the function ${\cal X}(u)$ can be written as
$$
{\cal X}(u)=\sum_{i+j\ge 1} \gamma_{i,j} u^i (u^\lambda)^j
$$
and let $\Delta({\cal X}):=\{(i,j)\in\mathbb{R}_+^2:\gamma_{i,j}\ne 0\}$.
For the rational subcase,
i.e., $\lambda=p/q$ for a pair of coprime positive integers $p$ and $q$,
we let $w=u^{1/q}$ and then equation \eqref{equ:xX} becomes
$$
x-{\cal X}(w^q)=x-\sum_{(i,j)\in\Delta({\cal X})} \gamma_{i,j}w^{iq+jp}=0,
$$
an analytic equation in $x$ and $w$.
Further the Newton-Puiseux Theorem (\cite[Theorem 4.2.7]{KP02})
yields that the solution $w=W(x)$ of the above equation is a convergent Puiseux series
and therefore the solution $u=w^q=(W(x))^q$ of equation \eqref{equ:xX}
is also a convergent Puiseux series,
which satisfies case {\bf(ii)} of Theorem~\ref{thm:O}.

For the irrational case,
we have $i+\lambda j\ne i'+\lambda j'$ for any two distinct lattice points $(i,j)$ and $(i',j')$.
So lattice points in the set $\Delta({\cal X})$ can be arranged into a sequence
$((i_1,j_1),(i_2,j_2),\cdots)$ such that
$
i_1+\lambda j_1<i_2+\lambda j_2<i_3+\lambda j_3<\cdots.
$
Then equation \eqref{equ:xX} can be rewritten as
$$
x-{\cal X}(u)=x-\sum_{k\ge 1}\gamma_{i_k,j_k} u^{\delta_k}=0,
$$
where $\delta_k:=i_k+\lambda j_k$.
Substituting the variable change $u=x^{1/\delta_1}(z_*+z)$ with
$z_*:=(\gamma_{i_1,j_1})^{-1/\delta_1}$
in the above equation,
we obtain that
\begin{align}
0=x-\sum_{k\ge 1}\gamma_{i_k,j_k}\{x^{1/\delta_1}(z_*+z)\}^{\delta_k}
=xF(x,z),
\end{align}
where
$$
F(x,z):=1-\gamma_{i_1,j_1}(z_*+z)^{\delta_1}
-\sum_{k\ge 2}\gamma_{i_k,j_k}x^{\delta_k/\delta_1-1}(z_*+z)^{\delta_k}.
$$
Note that
$F(0,0)=1-\gamma_{i_1,j_1}z_*^{\delta_1}=0$ and
$F_z(0,0)=-\delta_1 (\gamma_{i_1,j_1})^{1/\delta_1}\ne 0$.
By the Implicit Function Theorem and the method of undetermined coefficients,
the equation $F(x,z)=0$ has the solution of the form
$$
z=Z(x)=e_1x^{\mu_1}+e_2x^{\mu_2}+\cdots+e_k x^{\mu_k}+\cdots,
$$
where $(\mu_1,\mu_2,\dots,\mu_k,\cdots)$ is a strictly increasing sequence of positive real numbers tending to infinity.
Then the solution
\begin{align}
u={\cal X}^{-1}(x)
=x^{1/\delta_1}(z_*+Z(x))
=x^{1/\delta_1}(z_*+\sum_{k\ge 1}e_k x^{\mu_k})
\label{C2X-1}
\end{align}
of equation \eqref{equ:xX} satisfies case {\bf(ii)} of Theorem~\ref{thm:O}.

In case \textbf{(C3)},
equation \eqref{equ:xX} is of the form
$$
x-{\cal X}(u)=x-c_k u^k-\sum_{i\ge k+1} u^i P_i(\ln u)=0.
$$
We substitute $x=X^k$ and $u=X(w_*+w)$ with $w_*:=(c_k)^{-1/k}$
in the above equation and then it becomes $X^kG(w,X)=0$,
where
\begin{align*}
G(w,X)&:=1-c_k(w_*+w)^k-\sum_{i\ge k+1} X^{i-k} (w_*+w)^i P_i\big(\ln X+\ln(w_*+w)\big)
\\
&=-c_k\left(kw_*^{k-1}w+\cdots+w^k\right)+\sum_{i\ge 0, j\ge 1}w^i X^j {\cal P}_{i,j}(\ln X)
\end{align*}
for some polynomials ${\cal P}_{i,j}$\,s.
Note that $G(0,0)=0$ and $G_w(0,0)=-k(c_k)^{1/k}\ne 0$.
By the Implicit Function Theorem and the method of undetermined coefficients,
the equation $G(w,X)=0$ has the solution of the form
$$
w=W(X)=\sum_{s\ge 1}X^s Q_s(\ln X)
$$
for some polynomial $Q_s$\,s.
It follows that the solution
\begin{align}
u={\cal X}^{-1}(x)=x^{\frac{1}{k}}\left(w_*+W(x^{\frac{1}{k}})\right)
=w_*x^{\frac{1}{k}}+\sum_{s\ge 1}x^{\frac{s+1}{k}}
Q_s\left(\frac{\ln x}{k}\right)
\label{C3X-1}
\end{align}
of equation \eqref{equ:xX}
satisfies case {\bf(iii)} of Theorem~\ref{thm:O}.

In case \textbf{(C4)},
equation \eqref{equ:xX} is of the form
$$
x-{\cal X}(u)=x-u^k P_k(\ln u)-\sum_{i\ge k+1} u^i P_i(\ln u)=0.
$$
Assume that $\deg P_i=\rho_i$ and
$P_i(\omega):=e_{i,\rho_i}\omega^{\rho_i}+\cdots+e_{i,1}\omega+e_{i,0}$
for all $i\ge k$.
Under the variable changes $x=X^k$ and $u=X\Phi(X)(u_0+Z)$ with
$$
\Phi(X):=(-\ln X)^{-\frac{\rho_k}{k}}~~~\mbox{and}~~~
u_0:=\left((-1)^{\rho_k}e_{k,\rho_k}\right)^{-\frac{1}{k}},
$$
the above equation becomes $X^kH(Z,X,\Phi)=0$,
where $H(Z,X,\Phi):=1-\{I\}-\{II\}$,
\begin{align*}
\{I\}&:=\Phi^k(u_0+Z)^kP_k\big(\ln X+\ln \Phi+\ln(u_0+Z)\big),
\\
\{II\}&:=\sum_{i\ge k+1}X^{i-k}\Phi^i(u_0+Z)^iP_i\big(\ln X+\ln \Phi+\ln(u_0+Z)\big).
\end{align*}
One can compute that
\begin{align*}
\{I\}=\Phi^k\sum_{j\ge 0} Q_{k,j}(\ln X,\ln \Phi)Z^j,
\end{align*}
where $Q_{k,j}$\,s are all polynomials with
$\deg Q_{k,j}\le\rho_k$ for all $j\ge 0$
and moreover
\begin{align*}
Q_{k,0}(\ln X,\ln \Phi)
&=u_0^kP_k\big(\ln X\!+\!\ln \Phi\!+\!\ln(u_0)\big)
\\
&=u_0^ke_{k,\gamma_k}(\ln X)^{\gamma_k}+R_0(\ln X,\ln \Phi)
\\
&=(\ln X)^{\gamma_k}+R_0(\ln X,\ln \Phi),
\\
Q_{k,1}(\ln X,\ln \Phi)&=u_0^{k-1}P'_k\big(\ln X\!+\!\ln \Phi\!+\!\ln(u_0)\big)
+ku_0^{k-1}P_k\big(\ln X\!+\!\ln \Phi\!+\!\ln(u_0)\big)
\\
&=ku_0^{k-1}e_{k,\rho_k}(\ln X)^{\rho_k}+R_1(\ln X,\ln \Phi)
\\
&=\frac{k}{u_0}(\ln X)^{\rho_k}+R_1(\ln X,\ln \Phi),
\end{align*}
and those remainder parts
$R_0(\ln X,\ln \Phi)$ and $R_1(\ln X,\ln \Phi)$
satisfy that
\begin{align*}
\lim_{X\to 0^+}\Phi^k R_0(\ln X,\ln \Phi)
=\lim_{X\to 0^+}\Phi^k R_1(\ln X,\ln \Phi)
=0.
\end{align*}
It follows that
\begin{align*}
\{I\}=1&+\Phi^k R_0(\ln X,\ln \Phi)
+\left((-1)^{\rho_k}\frac{k}{u_0}+\Phi^k R_1(\ln X,\ln \Phi)\right) Z
\\
&+\sum_{j\ge 2} \Phi^k Z^j Q_{k,j}(\ln X,\ln \Phi).
\end{align*}
Similarly,
one can compute that
\begin{align*}
\{II\}=\sum_{i\ge k+1,j\ge 0} X^{i-k} \Phi^i Z^j Q_{i,j}(\ln X,\ln \Phi),
\end{align*}
where $Q_{i,j}$\,s are some polynomials.
Consequently,
$H(Z,X,\Phi)$ can be rewritten as
\begin{align*}
H(Z,X,\Phi)
=1-\{I\}-\{II\}
=(-1)^{\rho_k+1}\frac{k}{u_0}Z
-\sum_{i\ge k,j\ge 0} X^{i-k} \Phi^i Z^j \widetilde{Q}_{i,j}(\ln X,\ln \Phi),
\end{align*}
where
$\widetilde{Q}_{k,0}:=R_0(\ln X,\ln \Phi)$,
$\widetilde{Q}_{k,1}:=R_1(\ln X,\ln \Phi)$,
and $\widetilde{Q}_{i,j}:=Q_{i,j}(\ln X,\ln \Phi)$ for all other $i$ and $j$.
By the Implicit Function Theorem and the method of undetermined coefficients,
we find that the equation $H(Z,X,\Phi)=0$ has the solution of the form
$$
Z=Z(X,\Phi)=\sum_{i+j\ge 1} X^i \Phi^j {\cal Q}_{i,j}(\ln X,\ln \Phi)
$$
for some polynomial $P_{i,j}$\,s.
Thus equation \eqref{equ:xX} has the solution of the form
\begin{align}
u={\cal X}^{-1}(x)
&=x^{\frac{1}{k}}\Phi(x^{\frac{1}{k}})(u_0+Z(x^{\frac{1}{k}},\Phi(x^{\frac{1}{k}})))
\nonumber
\\
&=x^{\frac{1}{k}}\boldsymbol{\ell}^{\frac{\rho_k}{k}}
\left(\tilde{u}_0
+\sum_{i+j\ge 1} x^{\frac{i}{k}} \boldsymbol{\ell}^{\frac{\rho_k j}{k}}
\widetilde{{\cal Q}}_{i,j}(\ln x,\ln \boldsymbol{\ell})\right),
\label{C4X-1}
\end{align}
where $\tilde{u}_0$ is a nonzero constant,
$\widetilde{{\cal Q}}_{i,j}$\,s are polynomials and $\boldsymbol{\ell}:=-1/\ln x$.
Consequently,
the proof of Lemma~\ref{lm:inverse} is completed.
\qquad$\Box$

\subsection{Proof of Theorem~\ref{thm:O}}

{\bf Proof.}
We see from
the parametric representation \eqref{xxnyyn} that
the characteristic orbit $\Gamma$ has the explicit expression
$
y=y(x):={\cal Y}\circ{\cal X}^{-1}(x).
$
By Remark~\ref{rmk:Y},
we only need to consider the following 6 cases:
{\bf(S1)}   both ${\cal X}$ and ${\cal Y}$ satisfy {\bf(C1)},
{\bf(S2)}  both ${\cal X}$ and ${\cal Y}$ satisfy {\bf(C2)},
{\bf(S3)} both ${\cal X}$ and ${\cal Y}$ satisfy {\bf(C3)},
{\bf(S4)}  both ${\cal X}$ and ${\cal Y}$ satisfy {\bf(C4)},
{\bf(S5)}   ${\cal X}$ satisfies {\bf(C3)} but ${\cal Y}$ satisfies {\bf(C4)}, and
{\bf(S6)}  ${\cal X}$ satisfies {\bf(C4)} but ${\cal Y}$ satisfies {\bf(C3)}.

In case {\bf(S1)},
${\cal Y}(u)\in \mathbb{R}[\![u]\!]$ (or $\mathbb{R}\{u\}$)
and ${\cal X}^{-1}(x)\in \mathbb{R}[\![x^{1/n}]\!]$ (or $\mathbb{R}\{x^{1/n}\}$).
Thus
$y(x)={\cal Y}\circ{\cal X}^{-1}(x) \in \mathbb{R}[\![x^{1/n}]\!]$ (or $\mathbb{R}\{x^{1/n}\}$)
and satisfies {\bf(i)} (or {\bf(ii)}) of Theorem~\ref{thm:O}.
In case {\bf(S2)},
if $\lambda$ is rational,
then both ${\cal Y}(u)$ and ${\cal X}^{-1}(x)$ are convergent Puiseux series
and therefore $y(x)$ is also a convergent Puiseux series,
which satisfies {\bf(ii)} of Theorem~\ref{thm:O}.
On the contrary,
if $\lambda$ is irrational,
then ${\cal Y}(u)\in\mathbb{R}\{u,u^\lambda\}$ and
${\cal X}^{-1}(x)$ has the expansion \eqref{C2X-1}.
We see from \eqref{C2X-1} that $\left({\cal X}^{-1}(x)\right)^{i+\lambda j}$
is a series of the form given in {\bf(ii)} of Theorem~\ref{thm:O} for all $i+j\ge 1$.
Hence the expansion of $y(x)={\cal Y}\circ{\cal X}^{-1}(x)$
satisfies {\bf(ii)} of Theorem~\ref{thm:O}.
In case {\bf(S3)}, we have
$$
{\cal Y}(u)=c_su^s+\sum_{i\ge s+1} u^i P_i(\ln u)
$$
and ${\cal X}^{-1}(x)$ has the expansion \eqref{C3X-1}.
Note that
$$
\ln {\cal X}^{-1}(x)
=\frac{1}{k}\ln x+\ln\left(w_*+\sum_{s\ge 1}x^{\frac{s}{k}}
Q_s\left(\frac{\ln x}{k}\right)\right).
$$
Then $(u^i P_i(\ln u))|_{u={\cal X}^{-1}(x)}$
is a convergent fractional power series in $x$ with polynomial coefficients in $\ln x$
for all $i\ge s+1$
and therefore the expansion of $y(x)={\cal Y}\circ{\cal X}^{-1}(x)$ is also such a series,
which satisfies {\bf(iii)} of Theorem~\ref{thm:O}.
In case {\bf(S4)}, we have
$$
{\cal Y}(u)=u^s P_s(\ln u)+\sum_{i\ge s+1} u^i P_i(\ln u)
$$
and ${\cal X}^{-1}(x)$ has the expansion \eqref{C4X-1}.
We see from \eqref{C4X-1} that
\begin{align*}
\ln{\cal X}^{-1}(x)
=\frac{1}{k}\ln x+\frac{\rho_k}{k}\ln\boldsymbol{\ell}+\ln\tilde{u}_0
+\sum_{i+j\ge 1} x^{\frac{i}{k}} \boldsymbol{\ell}^{\frac{\rho_k j}{k}}
{\cal P}_{i,j}(\ln x,\ln \boldsymbol{\ell}).
\end{align*}
Then the expansion of $(u^i P_i(\ln u))|_{u={\cal X}^{-1}(x)}$
satisfies {\bf(iii)} of Theorem~\ref{thm:O} for all $i\ge s$
and therefore
the expansion of $y(x)={\cal Y}\circ{\cal X}(x)$ also
satisfies {\bf(iii)} of Theorem~\ref{thm:O}.
Similarly,
we can prove that expansions of $y(x)$ in cases {\bf(S5)} and {\bf(S6)}
both satisfy {\bf(iii)} of Theorem~\ref{thm:O}.
This completes the proof of Theorem~\ref{thm:O}.
\qquad$\Box$


\end{document}